\pdfoutput=1
\documentclass[DIV=14]{scrartcl}
\usepackage{amsmath,amsthm,amssymb}
\usepackage{adjustbox}
\usepackage{breqn,bm,ifpdf,authblk,relsize,extpfeil}
\usepackage{mathalfa}
\usepackage{eucal}
\usepackage[abbrev,nobysame,lite]{amsrefs}
\usepackage[dvipsnames]{xcolor}
\usepackage[T1]{fontenc}
\usepackage{lmodern}
\usepackage[inline,shortlabels]{enumitem}
\usepackage{microtype}
\usepackage{hyperref}

\allowdisplaybreaks[1]

\makeatletter
\DeclareOldFontCommand{\rm}{\normalfont\rmfamily}{\mathrm}
\DeclareOldFontCommand{\sf}{\normalfont\sffamily}{\mathsf}
\DeclareOldFontCommand{\bf}{\normalfont\bfseries}{\mathbf}
\DeclareOldFontCommand{\it}{\normalfont\itshape}{\mathit}

\DeclareFontFamily{U}{BOONDOX-calo}{\skewchar\font=45 }
\DeclareFontShape{U}{BOONDOX-calo}{m}{n}{
  <-> s*[1.05] BOONDOX-r-calo}{}
\DeclareFontShape{U}{BOONDOX-calo}{b}{n}{
  <-> s*[1.05] BOONDOX-b-calo}{}
\DeclareMathAlphabet{\mathcalboondox}{U}{BOONDOX-calo}{m}{n}
\SetMathAlphabet{\mathcalboondox}{bold}{U}{BOONDOX-calo}{b}{n}
\DeclareMathAlphabet{\mathbcalboondox}{U}{BOONDOX-calo}{b}{n}

\newtheorem{thm}{Theorem}[section]
\newtheorem{lemma}[thm]{Lemma}
\newtheorem{prop}[thm]{Proposition}
\newtheorem{cor}[thm]{Corollary}

\newtheorem{thmabc}{Theorem}
\newtheorem{propabc}[thmabc]{Proposition}

\theoremstyle{definition}

\newtheorem{rem}[thm]{Remark}
\newtheorem{question}[thm]{Question}

\numberwithin{equation}{section}
\allowdisplaybreaks[1]

\renewcommand{\le}{\leqslant}
\renewcommand{\ge}{\geqslant}

\def\emph{}
\DeclareTextFontCommand{\bfemph}{\bfseries}
\DeclareTextFontCommand{\itemph}{\itshape}
\def\emph{\bfemph}

\newcommand{\q}{\ensuremath{\mathcalboondox{q}}}

\newextarrow{\xbigtoto}{{20}{20}{20}{20}}
{\bigRelbar\bigRelbar{\bigtwoarrowsleft\rightarrow\rightarrow}}

\makeatletter
\def\blankfootnote{\xdef\@thefnmark{}\@footnotetext}

\newcommand*{\textlabel}[2]{%
  \edef\@currentlabel{#1}
  \phantomsection
  #1\label{#2}
}
\makeatother

\newcommand{\incl}{\hookrightarrow}
\newcommand{\xto}{\xrightarrow}

\newcommand{\cI}{\ensuremath{\boldsymbol{\mathcal{I}}}}
\newcommand{\cQ}{\ensuremath{\boldsymbol{\mathcal{Q}}}}
\newcommand{\cQx}{\ensuremath{\boldsymbol{\mathcal{Q}^{\mathsmaller{\mathsmaller{*}}}}}}
\newcommand{\cQo}{\ensuremath{\boldsymbol{\mathcal{Q}^{\mathsmaller{\mathsmaller{\flat}}}}}}

\newcommand{\cZ}{\ensuremath{\boldsymbol{\mathcal{Z}}}}
\newcommand{\cZx}{\ensuremath{\boldsymbol{\mathcal{Z}^{\mathsmaller{\mathsmaller{*}}}}}}
\newcommand{\cZo}{\ensuremath{\boldsymbol{\mathcal{Z}^{\mathsmaller{\mathsmaller{\flat}}}}}}

\newcommand{\fA}{\ensuremath{\mathfrak{A}}}
\newcommand{\fAx}{\ensuremath{\mathfrak{A}^{\mathsmaller{\mathsmaller{*}}}}}
\newcommand{\fAo}{\ensuremath{\mathfrak{A}^{\mathsmaller{\mathsmaller{\flat}}}}}

\newcommand{\fX}{\ensuremath{\mathfrak{X}}}
\newcommand{\fXx}{\ensuremath{\mathfrak{X}^{\mathsmaller{\mathsmaller{*}}}}}

\DeclareMathOperator{\ask}{ask}

\DeclareMathOperator{\ad}{ad}
\DeclareMathOperator{\Hom}{Hom}
\DeclareMathOperator{\End}{End}

\DeclareRobustCommand{\myhash}{\adjustbox{valign=B,totalheight=.57\baselineskip}{\#}}%

\newcommand{\opair}{\ensuremath{\!\diamond}}

\DeclareMathOperator{\cnt}{\myhash}
\DeclareMathOperator{\cntsymb}{\underline{\myhash}}

\newcommand{\Primes}{\ensuremath{\mathcal{P}}}
\newcommand{\PrimePowers}{\ensuremath{\mathcal{D}}}

\DeclareMathOperator{\concnt}{k}
\DeclareMathOperator{\asksymb}{\underline{\mathsf{ask}}}
\DeclareMathOperator{\consymb}{\underline{\mathsf{k}}}
\DeclareMathOperator{\orbsymb}{\underline{\mathsf{\myhash}}}

\newcommand{\imm}{\ensuremath{\rightarrowtail}}

\DeclareMathOperator{\Spec}{Spec}
\DeclareMathOperator{\diag}{diag}
\DeclareMathOperator{\GL}{GL}
\DeclareMathOperator{\Coker}{Coker}
\DeclareMathOperator{\Ker}{Ker}
\DeclareMathOperator{\Uni}{U}
\DeclareMathOperator{\Mat}{M}
\DeclareMathOperator{\rank}{rk}

\newcommand{\card}[1]{\lvert#1\rvert}
\DeclarePairedDelimiter{\abs}{\lvert}{\rvert}

\newcommand{\llb}{\ensuremath{[\![ }}
\newcommand{\rrb}{\ensuremath{]\!] }}

\newcommand{\divides}[2]{\ensuremath{ {#1} \mid {#2} }}

\renewcommand{\AA}{\mathbf{A}}
\newcommand{\QQ}{\mathbf{Q}}
\newcommand{\FF}{\mathbf{F}}
\newcommand{\ZZ}{\mathbf{Z}}

\newcommand{\CC}{\mathbf{C}}
\newcommand{\fg}{\ensuremath{\mathfrak g}}
\newcommand{\fn}{\ensuremath{\mathfrak n}}
 \newcommand{\sH}{\mathsf{H}}

\ifpdf
  \hypersetup{pdftitle={On the enumeration of orbits of unipotent groups over finite fields},
    pdfauthor={Tobias Rossmann}
  }
\fi

\title{On the enumeration of orbits of unipotent groups over finite fields}
\subtitle{\it ``Where the wild things are''}

\author{Tobias Rossmann}

\date{}

\begin{document}
\thispagestyle{empty}

\maketitle

\thispagestyle{empty}
\vspace*{-2em}
\begin{abstract}
  \small
  We show that the enumeration of linear orbits and conjugacy classes
  of $\ZZ$-defined unipotent groups over finite fields is ``wild'' in the following
  sense:
  given an arbitrary scheme $Y$ of finite type over $\ZZ$ and integer $n\ge 1$,
  the numbers $\cnt\! Y(\FF_q) \bmod q^n$ can be expressed, uniformly in~$q$, in
  terms of the numbers of linear orbits (or numbers of conjugacy classes) of
  finitely many $\ZZ$-defined unipotent groups over $\FF_q$ and finitely many
  Laurent polynomials in $q$.
\end{abstract}

\blankfootnote{\noindent{\itshape 2020 Mathematics Subject Classification.}
  20D15, 
  20E45, 
  05A15, 
  11G25  

}

\section{Introduction}
\label{s:intro}

\paragraph{Linear orbits and conjugacy classes.}
This article is devoted to what might be called the ``symbolic enumeration''
of linear orbits and conjugacy classes of families of matrix groups.
Let $\concnt(G)$ denote the number of conjugacy classes (or \emph{class number})
of a group $G$.
For a (commutative, unital) ring~$R$, let $\GL_n(R)$ denote the group of
invertible $n\times n$ matrices over $R$, and let $\Uni_n(R) \le \GL_n(R)$ be
the subgroup of upper unitriangular matrices.

It is well-known that $\concnt(\GL_n(\FF_q))$ is given by a polynomial in $q$;
see \cite[Exercise~1.190]{Sta12}.
Enumerating linear orbits (i.e.\ orbits of linear groups on their natural
modules) often turns out to be an easier task than
enumerating conjugacy classes.
For example, while proving that $\concnt(\GL_n(\FF_q))$ depends polynomially
on $q$ is elementary, it is a complete triviality that
$\cnt(\FF_q^n/\GL_n(\FF_q)) = 2$ for $n\ge 1$.
Whether $\concnt(\Uni_n(\FF_q))$ depends polynomially on $q$ or not is
precisely the subject of ``Higman's conjecture'';
see e.g.\ \cite{Hig60a,VLA03,PS15}.
Again, linear orbits are much more easily enumerated: it is an easy exercise
to show that $\cnt(\FF_q^n/\Uni_n(\FF_q)) = nq-n+1$.

\paragraph{Linear algebraic groups.}
The functors $\GL_n$ and $\Uni_n$ are (represented by) group
schemes over $\ZZ$.
In this article, by a \emph{linear algebraic group} $G$ over $\QQ$, we mean a
closed subgroup scheme of $\GL_n\otimes \QQ$ for some $n\ge 1$.
Equivalently, we may think of such a $G$ as being encoded by a
Zariski-closed subgroup of $\GL_n(\CC)$ whose vanishing ideal
is generated by polynomials with rational coefficients.
Given $G\le \GL_n\otimes \QQ$, by clearing denominators,
for each ring~$R$, we obtain a group $G(R) \le \GL_n(R)$.
(More formally, letting $\mathcal G$ denote the scheme-theoretic closure of
$G$ in $\GL_n$, a group scheme over $\ZZ$, we abuse notation and write $G(R)$
for $\mathcal G(R)$.)

It is well known that every unipotent algebraic group over $\QQ$ is isomorphic
to a closed subgroup scheme of $\Uni_n\otimes \QQ$ for some $n$;
see e.g.\ \cite[Ch.\ IV, \S 2, Prop. 2.5]{DG70}.
This article is a contribution to the following.

\begin{question}
  \label{qu:motivation}
  Given a unipotent algebraic group $G\le \Uni_n\otimes \QQ$ over $\QQ$, how
  do the numbers $\concnt(G(\FF_q))$ and $\cnt(\FF_q^n/G(\FF_q))$ vary with
  the prime power $q$?
\end{question}

Here, the underlying prime $p$ of $q = p^f$ is allowed to vary, as is the
exponent $f\ge 1$.
By the \emph{symbolic enumeration} of the orbits in 
Question~\ref{qu:motivation}, we mean the (arguably vaguely defined) task of
providing meaningful descriptions of the counting functions of $q$ given
there.
For example, the tamest case occurs when these counting functions are given by
polynomials in $q$, as they are e.g.\ for $G = \Uni_3\otimes \QQ$ (the
Heisenberg group).

\begin{rem}
  Assuming that we are willing to discard finitely many (potentially)
  exceptional characteristics in Question~\ref{qu:motivation}, it turns out that
  it suffices (in a precise sense) to consider class numbers and numbers of
  linear orbits over $\FF_p$ for \itemph{primes} $p$, rather than over all $\FF_q$ for
  \itemph{prime powers} $q$; see Lemma~\ref{lem:Nxp} and the comments that follow
  it.
\end{rem}

\paragraph{Geometric bounds.}
For an upper bound on the complexity of the counting functions of~$q$ in
Question~\ref{qu:motivation}, the orbit-counting lemma from elementary group
theory shows that understanding either one of $\concnt(G(\FF_q))$ and
$\cnt(\FF_q^n/G(\FF_q))$ is at most as hard as enumerating $\FF_q$-rational
points on schemes of finite type over~$\ZZ$; cf.\ Corollary~\ref{cor:oct}.

\paragraph{Main result: wildness mod $q^n$.}
Given that the study and enumeration of rational points on schemes is an
extraordinarily challenging and deep task, one might wonder whether the upper
bound on the complexity of the functions of $q$ in
Question~\ref{qu:motivation} provided by the preceding paragraph
truly reflects the actual difficulty of enumerating orbits of unipotent
groups.
Our main result (Theorem~\ref{thm:main_congruence}) shows that it does---after a fashion.

As explained in \cite[\S 2.4]{cico} (and partially recalled in
\S\ref{ss:classes_via_ask}), each alternating bilinear map $*\colon
\ZZ^d\times \ZZ^d \to \ZZ^e$ gives rise to an associated \itemph{Baer group
  scheme} $G_*$.
This is a unipotent group scheme of class at most $2$ with underlying scheme
$\AA^{d+e}$.
The group $G_*(\ZZ)$ is torsion-free, nilpotent of class at most $2$,
and it has Hirsch length $d+e$.

\begin{thmabc}
  \label{thm:main_congruence}
  Let $Y$ be a scheme of finite type over $\ZZ$.
  Let $n\ge 1$.
  Then there are finitely many commutative group schemes $M_1,\dotsc,M_r$ with $M_i
  \le \Uni_{d_i}$, Baer group schemes $G_1,\dotsc,G_r$, 
  and univariate Laurent polynomials
  $f_1(X),\dotsc,f_r(X),g_1(X),\dotsc,g_r(X) \in \ZZ[X^{\pm 1}]$ such that for
  each prime power $q$, the numbers 
  $F(q) := \sum\limits_{i=1}^r f_i(q) \concnt(G_i(\FF_q))$
  and $G(q) := \sum\limits_{i=1}^r g_i(q) \cnt(\FF_q^{d_i}/M_i(\FF_q))$
  are integers which satisfy
  \begin{equation}
    \cnt\! Y(\FF_q) \equiv F(q) \equiv G(q) \pmod {q^n}.
    \label{eq:main_congruence}
  \end{equation}
\end{thmabc}

Our proof of Theorem~\ref{thm:main_congruence} will combine a deep result
due to Belkale and Brosnan~\cite{BB03} and machinery developed in
\cite{ask,ask2,cico}. 
The author does not know if the congruences in~\eqref{eq:main_congruence} can be
elevated to honest equations (perhaps at the cost of considering a larger
ring of rational functions than $\ZZ[X^{\pm 1}]$ for the coefficients $f_i(X)$
and $g_i(X)$).
Theorem~\ref{thm:main_congruence} is however strong enough to establish that
the class numbers and numbers of linear orbits in
Theorem~\ref{thm:main_congruence} are as geometrically ``wild'' as possible.
Indeed, given any scheme $Y$ of finite type over $\ZZ$, there exists $n\ge 1$
such that $\cnt\! Y(\FF_q) < q^n$ for all prime powers $q$.
Equation~\eqref{eq:main_congruence} therefore allows us to recover
the precise number $\cnt\! Y(\FF_q)$ from expressions
involving Laurent polynomials in $q$ as well as the class numbers
$\concnt(G_i(\FF_q))$ and numbers of linear orbits
$\cnt(\FF_q^{d_i}/M_i(\FF_q))$.

\begin{rem}
  Generalising work of Boston and Isaacs~\cite[Thm~3.2]{BI04},
  O'Brien and Voll~\cite[Thm~4.1]{O'BV15} showed that every sufficiently
  generic Pfaffian hypersurface $H$ over $\ZZ$ gives rise to a unipotent group
  scheme $G$ such that $\concnt(G(\FF_q))$ depends in an explicit form on
  $\cnt\! H(\FF_q)$ (for odd $q$).
  All explicit examples of unipotent group schemes $G$ for which
  $\concnt(G(\FF_q))$ depends ``geometrically'' on $q$ known to the author are
  essentially of this form.
\end{rem}

\begin{rem}
  While this will play no role in our proof of Theorem~\ref{thm:main_congruence},
  we note that said theorem reduces to the case that $Y$ is affine.
  That is, we may assume that $Y =
  \Spec(\ZZ[X_1,\dotsc,X_\ell]/(f_1,\dotsc,f_r))$, where $f_1,\dotsc,f_r\in
  \ZZ[X_1,\dotsc,X_\ell]$.
  In that case,
  $\cnt\! Y(\FF_q) = \cnt\{ x\in \FF_q^\ell : f_1(x) = \dotsb = f_r(x) = 0\}$.

  To reduce Theorem~\ref{thm:main_congruence} to the affine case, we
  may proceed as in the proof of \cite[Prop.\ 3.5.3]{CLNS18}.
  Namely, suppose that Theorem~\ref{thm:main_congruence} holds for affine $Y$.
  By subtraction, it then also holds in the quasi-affine case.
  For general $Y$, consider a finite cover $Y = \bigcup_{i\in I} Y_i$ in which
  each $Y_i$ is an open affine subscheme of $Y$.
  Writing $Y_J = \bigcap_{j\in J} Y_j$ for $\emptyset\not= J \subset I$,
  each $Y_J$ is then quasi-affine and the claim follows
  from the inclusion-exclusion principle.
\end{rem}

\begin{rem}
  In our proof of Theorem~\ref{thm:main_congruence},
  the group schemes $M_1,\dotsc,M_r$ and $G_1,\dotsc,G_r$
  will generally depend on $n$.
  However, we will see in Corollary~\ref{cor:r} that the number $r$ can be
  chosen to be independent of $n$.
\end{rem}

\paragraph{Related counting problems.}
There are numerous natural enumerative
problems in which, for each prime $p$ or prime power $q$, we seek to count
objects up to a suitable equivalence relation.
For example, given a group scheme $G$ acting on a scheme $X$, we may consider
the numbers $\card{X(\FF_q)/G(\FF_q)}$; this is the theme of the present
paper.
Other counting problems in the literature are of a less (obviously)
geometric origin.
In particular, let $f(n,p)$ denote the number of isomorphism classes of groups
of order $p^n$.
The study of $f(n,p)$ as a function of $p$ (for fixed $n$) is precisely the
subject of Higman's famous and open \emph{PORC conjecture}~\cite{Hig60b}:
it predicts that for fixed $n$, the function $p\mapsto f(n,p)$ is polynomial
on residue classes (relative to a suitable modulus).
This conjecture has been verified for $n\le 7$; see \cite{O'BVL05} as well as
\cite{BEO'B02} and the references therein.

Suspicions that Higman's PORC conjecture might be false have been expressed,
see~\cite{dSVL12}, in particular.
In \cite{dSVL12}, this was based on the authors's prediction that
certain non-PORC functions are ``extremely unlikely'' to add up to PORC 
functions.
(We note that while \cite{cico} is concerned with very different enumerative
problems, \cite[Cor.\ 1.2]{cico} and \cite[Thm~0.5]{BB03} together show that
thanks to cancellations, natural examples of highly non-PORC functions can
indeed add up to polynomials.)
Du Sautoy~\cite[Cor.\ 1.9]{dS00} showed that for fixed $n$, the
number $f(n,p)$ is expressible in terms of the numbers of $\FF_p$-rational
points of schemes over $\ZZ$.
One might well suspect that this (rather than the PORC conjecture) captures
the true general shape of the functions $p\mapsto f(n,p)$.

While Theorem~\ref{thm:main_congruence} does not constitute a contribution
towards resolving Higman's PORC conjecture either way, it does suggest
that natural group-theoretic counting problems can indeed be as wild
as arithmetic geometry allows;
cf.\ \cite[Problem 4.3]{dSG06}.

\paragraph{Low nilpotency class dominates.}
Theorem~\ref{thm:main_congruence} allows us to approximate the numbers $\cnt\!
Y(\FF_q)$ in terms of quantities derived from unipotent group schemes of class at
most~$2$.
As we will now indicate, it is perhaps no coincidence that higher nilpotency
class plays no role here.
Indeed, let $G \le \Uni_n\otimes \QQ$ be a unipotent algebraic group over $\QQ$.
Then the following proposition shows that (up to an increase of dimensions and the
exclusion of finitely many exceptional characteristics), the enumeration
of the linear orbits and conjugacy classes of $G(\FF_q)$ reduces to the case
that $G$ has nilpotency class at most $2$.

\begin{propabc}
  \label{prop:reduction}
  Let $G \le \Uni_n\otimes \QQ$ be a unipotent algebraic group of dimension
  $d$ over $\QQ$.
  \begin{enumerate}[(i)]
  \item
    \label{prop:reduction1}
    There exist a commutative algebraic group $M \le \Uni_{2n}\otimes \QQ$
    of dimension $d$ over $\QQ$ and an integer $u \ge 1$ such that
    $\cnt(\FF_q^n/G(\FF_q)) = q^{-n} \cnt(\FF_q^{2n}/M(\FF_q))$
    for all prime powers $q$ with $\gcd(q,u) = 1$.
  \item
    \label{prop:reduction2}
    \textup{(Cf.\ \cite[Prop.\ 2.5]{cico}).}
    There exist a unipotent algebraic group $H$
    of class at most $2$ and dimension $3d$ over $\QQ$ and an integer $v\ge 1$
    such that $\concnt(G(\FF_q)) = q^{-d} \concnt(H(\FF_q))$
    for all prime powers $q$ with $\gcd(q,v) = 1$.
  \end{enumerate}
\end{propabc}

Proposition~\ref{prop:reduction}\ref{prop:reduction2} follows from
\cite[Prop.\ 2.5]{cico}, a result pertaining to class-counting zeta
functions of unipotent group schemes.
One motivation for writing this article was to illustrate how some of the
machinery surrounding ``ask zeta functions'' developed in \cite{ask,ask2,cico}
can be applied to enumerative problems involving groups over finite fields.

\paragraph{Overview.}
In \S\ref{s:arithmetic_1}, we construct rings which we will use to
investigate counting problems depending on prime powers $q$.
The counting problems that we consider give rise to elements of rings $\cQ$
and~$\cQo$.
The former of these rings embeds as a dense subring into a certain topological
ring $\cQ_\q$.
In \S\ref{s:ask}, we mostly collect known material related to average sizes of
kernels and we explain how this relates to the enumeration of orbits of
unipotent groups.
This, in particular, gives rise to chains of submodules of $\cQ$ and $\cQ_\q$.
Finally, in \S\ref{s:proof}, we prove Theorem~\ref{thm:main_congruence}.
In fact, we deduce the latter theorem from Theorem~\ref{thm:main_density}
which asserts that several of the aforementioned submodules of $\cQ_\q$
all have the same~closure.

\subsection*{\textit{Notation}}

The symbol ``$\subset$'' indicates not necessarily proper inclusion.
Rings are assumed to be commutative and unital.
Schemes (including group schemes) are assumed to be separated and of finite
type over their bases, the default of the latter being $\ZZ$.
Maps and groups usually act on the right.
Throughout, $\Primes$ denotes the set of all rational primes and $\PrimePowers
= \{ p^e : p\in\Primes, e\ge 1\}$ the set of prime powers.

\section{Rings from counting problems}
\label{s:arithmetic_1}

This section takes inspiration from the combinatorial motives of Belkale and
Brosnan~\cite{BB03} and Huber's affinoid rings~\cite{Hub96}.

\subsection[Constructing Z and Q]{Constructing $\cZx$ and $\cQx$}

Define $\cZ$ to be the subring of the product $\prod_\PrimePowers \ZZ$
consisting of elements $x\in \prod_\PrimePowers \ZZ$ 
that are polynomially bounded in the sense that given $x$,
there exists a constant $d\ge 0$ with $\abs{x_q} \le q^d$ for all
$q\in\PrimePowers$.
For example, if $X$ is a scheme,
then $\cntsymb\! X := (\card{X(\FF_q)})_{q\in\PrimePowers}$ belongs to $\cZ$.
(Recall that schemes are assumed to be of finite type.)
Write $\q = \cntsymb\!\AA^1 = (q)_{q\in\PrimePowers} \in \cZ$.
If $F$ is any set-valued functor on the category of rings which
is either a sub- or quotient functor of $X(\cdot)$ for a scheme $X$, then 
$\cntsymb\! F := (\card{F(\FF_q)})_{q\in\PrimePowers} \in \cZ$.
For example, if $G$ is a group scheme acting on a scheme~$X$, then the naive
orbit space functor $X/G$ given by $R\leadsto X(R)/G(R)$ yields an element
$\cntsymb(X/G) = (\card{X(\FF_q)/G(\FF_q)})_{q\in\PrimePowers} \in \cZ$.
As a special case, for any group scheme $G$, we write $\consymb (G) =
(\concnt(G(\FF_q)))_{q\in\PrimePowers}\in \cZ$.

\paragraph{Discarding small primes.}
We will often ignore finitely many potentially exceptional characteristics.
To that end, consider the ideal $\cI$ of $\cZ$ given by
\[
  \cI = \Bigl\{ x\in \cZ : \exists N \ge 1\, \forall q\in \PrimePowers\, \bigl(\gcd(q,N)
  = 1 \to x_q = 0\bigr)\Bigr\}.
\]
Define $\cZo = \cZ/\cI$.
Hence, $x,y\in \cZ$ are identified in $\cZo$ if and only if
$x_{p^e} = y_{p^e}$ for almost all (i.e.\ for all but finitely many)
primes $p\in\Primes$ and all $e \ge 1$.
By abuse of notation, we often use the same symbol for an element of $\cZ$ and
for its image in~$\cZo$.
Henceforth, a superscript ``$*$'' refers to either ``$\flat$'' or the
absence of a decoration, the convention being that ``$*$'' has the
same meaning within a given statement.
For instance, $\q\in \cZx$ is algebraically independent over $\ZZ$
and $\cZx$ embeds into the localisation $\cQx := \cZx[\q^{-1}] =
\cZx\otimes_{\ZZ[\q]}\ZZ[\q^{\pm 1}]$.
Note that we may view $\cQ$ as a subring of $\prod_\PrimePowers \QQ$.

\begin{lemma}
  $\bigcap_{n=1}^\infty \q^n\cZx = 0$.
\end{lemma}
\begin{proof}
  First note that $\bigcap_{n=1}^\infty \q^n\cZ \subset \bigcap_{n=1}^\infty
  \prod_{q\in \PrimePowers} q^n\ZZ = 0$.
  Let $x\in \cZ$ and suppose that $x + \cI \in \bigcap_{n=1}^\infty \q^n\cZo$.
  Since $x\in\cZ$, there exists $d\ge 1$ such that $\abs{x_q} \le q^d$ for all
  $q \in \PrimePowers$.
  As $x + \cI \in \q^{d+1}\cZo$, there exists $N\ge 1$ such that
  $\divides{q^{d+1}} {x_q}$ for all $q\in\PrimePowers$ with $\gcd(q,N) = 1$.
  We conclude that $x_q = 0$ whenever $\gcd(q,N) = 1$ whence $x \in \cI$.
  Thus, $\bigcap_{n=1}^\infty \q^n\cZo = 0$.  
\end{proof}

\subsection{Integral forms and the enumeration of orbits}

Let $X_0$ be a scheme over $\QQ$.
Choose an integral form $X$ of $X_0$.
That is, choose a scheme $X$ with $X\otimes \QQ \approx X_0$ over $\QQ$.
(The existence of such an $X$ follows by \itemph{spreading out};
see e.g.\ \cite[\S 3.2]{Poo17}.)
While $\cntsymb\! X \in \cZ$ depends on the choice of~$X$, the image of
$\cntsymb\! X$ in $\cZo$ only depends on $X_0$.
(Indeed, another spreading out argument shows that given integral forms $X$
and $X'$ of $X_0$, we obtain an isomorphism $X\otimes
\ZZ[1/N]\approx X'\otimes \ZZ[1/N]$ over $\ZZ[1/N]$ for some $N\ge 1$.
We then have $\card{X(\FF_q)} = \card{X'(\FF_q)}$ for all $q\in\PrimePowers$
with $\gcd(q,N) = 1$
whence $\cntsymb\! X = \cntsymb\! X'$ in $\cZo$.)
We let $\cntsymb\! X_0$ denote the image of $\cntsymb\! X$ in $\cZo$.

Let $G_0 \leqslant \GL_n\otimes \QQ$ be a linear algebraic group over $\QQ$.
The scheme-theoretic closure $G$ of $G_0$ in $\GL_n$ is an integral form of
$G_0$ as a group scheme; see e.g.\ \cite[\S 1]{Fom97}.
Define $\cntsymb(\AA^n_\QQ / G_0)$
(resp.\ $\consymb(G_0)$) to be the 
image of $\cntsymb(\AA^n/G)$ (resp.\ $\consymb(G)$) in $\cZo$.
Note that $\consymb(G_0)$ does not depend on the choice of the embedding of
$G_0$ into $\GL_n\otimes \QQ$.

\subsection{Counting orbits: geometric perspectives}

Let $G$ be a group scheme acting on a scheme $X$.
While the naive orbit space functor $X/G$ is hardly ever
(represented by) a scheme, the point-counting function $q\mapsto
\card{X(\FF_q)/G(\FF_q)}$ \textit{does} behave somewhat similarly to the one
attached to a scheme.
To see this, let $X \opair G$ be the closed subscheme of $X\times G$
representing $R \leadsto \{ (x,g) \in X(R)\times G(R) : xg = x\}$.

\begin{lemma}
  \label{lem:oct}
  $\card{X(R)/G(R)}
  = \card{(X \opair G)(R)} \cdot {\card{G(R)}}^{-1}$
  for each finite ring $R$.
\end{lemma}
\begin{proof}
  This follows immediately from the orbit-counting lemma.
\end{proof}

\begin{cor}
  \label{cor:oct}
  Let $G$ be a group scheme acting on a scheme $X$.
  Suppose that $G\otimes \QQ$ is a
  unipotent algebraic group of dimension~$d$ over~$\QQ$.
  Then $\cntsymb(X/G) = \cntsymb(X \opair G) \cdot \q^{-d}$ in~$\cQo$.
\end{cor}
\begin{proof}
  $G\otimes \QQ \approx \AA^d_{\QQ}$ as schemes over $\QQ$; cf.\ \cite[Ch.\ IV, \S 2, Prop.\ 4.1]{DG70}.
\end{proof}

When our objective is to enumerate orbits of group scheme actions on the level
of rational points over finite fields, assuming that we are willing to ignore
finitely many exceptional characteristics, then it often turns out to be
enough to consider prime fields.
More formally, inspired by properties of the point-counting functions
$p\mapsto \card{Y(\FF_p)}$ attached to a scheme~$Y$ as described in
\cite[Ch.\ 1]{Ser12}, we observe the following.

\begin{lemma}
  \label{lem:Nxp}
  For each $i=1,\dotsc,r$,
  let $G_i$ be a group scheme acting on a scheme $Y_i$.
  Let $f_1(X),\dotsc,f_r(X) \in \QQ[X^{\pm 1}]$.
  For $q\in \PrimePowers$, define $F(q) = \sum\limits_{i=1}^r
  \card{Y_i(\FF_q)/G_i(\FF_q)} \cdot f_i(q)$.
  Suppose that $F(p) = 0$ for all $p\in P$,
  where $P \subset \Primes$ has natural density $1$.
  Then there exists $N\ge 1$ such that $F(q) = 0$ for all $q\in\PrimePowers$ with
  $\gcd(q,N) = 1$.
  That is, $(F(q))_{q\in\PrimePowers}$ vanishes in $\cQo$.
\end{lemma}
\begin{proof}
  Let $Z_i = (Y_i\,\opair G_i) \times \prod_{j\not= i} G_j$
  and $\tilde F(q) = \sum_{i=1}^r \card{Z_i(\FF_q)} f_i(q)$.
  By Lemma~\ref{lem:oct} and since $G_i(\FF_q)\not= \emptyset$,
  for each $q\in \PrimePowers$, we have $F(q) = 0$ if and only if
  $\tilde F(q) = 0$.
  It thus suffices to establish the analogous claim for $\tilde F$.
  This follows by combining Chebotarev's density theorem and Grothendieck's
  trace formula; cf.~\cite[Thm 2.3]{stability}.
\end{proof}

We record two particular consequences:
\begin{enumerate}
\item
  Let $G_i$ act on $Y_i$ for $i = 1,2$.
  Suppose that $\card{Y_1(\FF_p)/G_1(\FF_p)} = \card{Y_2(\FF_p)/G_2(\FF_p)}$
  for almost all primes $p$.
  Then in fact $\card{Y_1(\FF_q)/G_1(\FF_q)} = \card{Y_2(\FF_q)/G_2(\FF_q)}$
  for all prime powers $q$ except perhaps for a finite number of exceptional
  characteristics.
  (For trivial actions, this is \cite[Thm~1.3]{Ser12}.)
\item
  Let $G$ act on $Y$ and suppose that
  $\card{Y(\FF_p)/G(\FF_p)}$ is given by a (rational) polynomial in~$p$ for
  almost all primes $p$.
  Then $\card{Y(\FF_q)/G(\FF_q)}$ too depends polynomially on $q$ except perhaps
  for a finite number of exceptional characteristics. 
\end{enumerate}

\begin{rem}
  For related geometric perspectives on the enumeration of conjugacy classes,
  see \cite[\S 2]{Rob98} and \cite[\S 2]{dS05}.
\end{rem}

\subsection[Completing Z and Q]{Completing $\cZ$ and $\cQ$}
\label{ss:topology}

Let  $\cZ_\q = \varprojlim \cZ/\q^n\cZ$ be the $\q$-adic completion of~$\cZ$.
As $\bigcap_{n=1}^\infty \q^n\cZ = 0$, we may regard $\cZ$ as a (dense)
subring of~$\cZ_\q$. 
Since $\q$ is a regular element of $\cZ_\q$, the latter ring embeds into
$\cQ_\q := \cZ_\q[\q^{-1}]$.
We regard $\cQ$ as a subring of $\cQ_\q$.
We topologise $\cQ_\q$ via the identification $\cQ_\q = \varinjlim (\cZ_\q
\xto\q \cZ_\q \xto\q \dotsb)$.
This turns $\cQ_\q$ into a topological ring which contains $\cZ_\q$ as an
open and closed subring.
The subspace topology of $\cZ_\q$ within $\cQ_\q$ coincides with our original
$\q$-adic topology on $\cZ_\q$.

As usual, for $p\in\Primes$, let $\ZZ_p$ and $\QQ_p$ denote the ring of
$p$-adic integers and field of $p$-adic numbers, respectively\footnote{Our
  construction of $\cZ_\q$ and $\cQ_\q$ is an imitation of the usual
construction of $\ZZ_p$ and $\QQ_p$, respectively.}.
For $e\ge 1$, we write $\ZZ_{p^e} = \ZZ_p$ and $\QQ_{p^e} = \QQ_p$;
note that $\ZZ_q = \varprojlim \ZZ/q^n\ZZ$ for $q\in\PrimePowers$.
Let $\cZ_n = \prod_{q\in\PrimePowers} \ZZ/q^n\ZZ$, regarded as a discrete ring.
As $\cZ_n \approx \cZ/\q^n$ naturally, we may identify $\cZ_\q =
\varprojlim \cZ_n$ as topological rings.
Since products commute with inverse limits (both being
categorical limits),
the natural embeddings $\cZ \incl \prod_\PrimePowers \ZZ$ and
$\cQ\incl \prod_\PrimePowers\QQ$ extend to a continuous isomorphism of
abstract rings (not a homeomorphism!)
$\cZ_\q \to \prod_{q\in\PrimePowers} \ZZ_q$ and a continuous ring monomorphism
$\cQ_\q \to \prod_{q\in\PrimePowers} \QQ_q, \,x \mapsto (x_q)_{q\in\PrimePowers}$,
respectively.

\begin{prop}
  $\cQ \cap \cZ_\q = \cZ$.
\end{prop}
\begin{proof}
  Let $x \in \cQ\cap \cZ_\q$.
  Then there exists $m\ge 0$ and $y\in \cZ$ such that $x_q = y_q/q^m$ for all
  $q\in \PrimePowers$.
  As $x_q\in \ZZ_q$, we conclude that $x_q\in \ZZ$ for all
  $q\in\PrimePowers$.
  The converse is clear.
\end{proof}

We may view $\cZ_\q$ as $\prod_{q\in\PrimePowers}\ZZ_q$ endowed with a
suitable topology of uniform convergence:

\begin{lemma}
  \label{lem:closure}
  Let $U \subset \cQ_\q$ and $x\in \cZ_\q$.
  Then $x\in\bar U$ if and only if for each $n\ge 1$, there exists $u\in U\cap
  \cZ_\q$ such that $x_q \equiv u_q \pmod{q^n}$ for all $q\in \PrimePowers$.
\end{lemma}
\begin{proof}
  Since $\cZ_\q$ is both open and closed in $\cQ_\q$,
  we have $\bar U \cap \cZ_\q = \overline{U \cap \cZ_\q}$;
  cf.\ \cite[Ch.~I, \S 1, no.\ 6, Prop.\ 5]{Bou98a}.
  Let $V_n$ denote the image of $U\cap \cZ_\q$ in the discrete space
  $\cZ_n$.
  The claim follows since $\overline{U\cap \cZ_\q} = \varprojlim V_n$;
  see \cite[Ch.~I, \S 4, no.\ 4, Cor.]{Bou98a}.
\end{proof}

It is easy to see that $\ZZ$ is a closed and discrete subring of $\cZ_\q$
(hence of $\cQ_\q$).
Moreover, the subspace topology of $\ZZ[\q]$ in $\cZ_\q$ (and in $\cQ_\q$) is
$\q$-adic, and we may identify
$\overline{\ZZ[\q]} = \ZZ\llb \q\rrb = \varprojlim \ZZ[\q]/\q^n$ within~$\cZ_\q$.

\begin{rem}
  \begin{enumerate*}
  \item
    While we do not need this here, the same steps as above allow us to define
    $\cZo_\q$ as the $\q$-adic completion of $\cZo$, and then construct
    $\cQo_\q = \cZo[\q^{-1}]$.
    One can show that the natural map $\cQ_\q \to \cQo_\q$ (induced by the
    quotient map $\cZ \to \cZo$) is continuous, open, and surjective.
    A reader might wonder why we chose to define $\cZ$ as a particular subring
    of $\prod_\PrimePowers\ZZ$ instead of simply defining $\cZ =
    \prod_\PrimePowers\ZZ$.
    (Both choices give rise to the same $\q$-adic completion.)
    Our definition of $\cZ$ ensures that $\cZo$ embeds into $\cZo_\q$.
  \item
    The pair $(\cQ_\q,\cZ_\q)$ is an example of a (non-Noetherian)
    \itemph{affinoid ring}; see \cite[\S 1.1]{Hub96} or \cite[\S\S
    2.1--2.2]{SW20} for background.
    In particular, our construction above is a special case of a much more
    general completion procedure.
  \end{enumerate*}
\end{rem}

\section{Average sizes of kernels and the enumeration of orbits}
\label{s:ask}

In \S\S\ref{ss:mreps}--\ref{ss:classes_via_ask},
we mostly collect known material from \cite{ask,ask2,cico} pertaining to the
enumeration of linear orbits and conjugacy classes of unipotent groups.
We then derive Proposition~\ref{prop:reduction} in \S\ref{ss:proof_reduction}.
Finally, in~\S\ref{ss:inclusions}, we deduce equations and inclusions between
various submodules of $\cQx$; these relations will play a crucial role in our
proof of Theorem~\ref{thm:main_congruence}.

\subsection{Module representations}
\label{ss:mreps}

For more on the following, see \cite[\S \S 2, 4, 7]{ask2}.
Let $R$ be a ring.
A \emph{module representation} over $R$ is a homomorphism
$M\xto\theta\Hom(V,W)$, where $M$, $V$, and $W$ are $R$-modules.
Equivalently, $\theta$ is determined by the bilinear map
$*\colon V\times M \to W$ given by $x * a = x(a\theta)$ for $a\in M$
and $x\in V$.
(Recall that we usually write maps on the right.)
Given a (commutative, associative, unital) $R$-algebra $S$, we obtain a module
representation $M\otimes_R S\xto{\theta^S} \Hom(V\otimes_R S,W\otimes_R S)$
over $S$ derived from $\theta$ by base change.
For $m\ge 0$, let ${}^m\theta$ denote the module
representation $M\to \Hom(V^{\oplus m},W^{\oplus m})$ over $R$ with $a({}^m\theta) =
(a\theta)^{\oplus m}$ for $a\in M$;
we identify ${}^m({}^n\theta) = {}^{mn}\theta$.

Let $(\cdot)^* = \Hom(\cdot, R)$.
A module representation $\theta$ over $R$ as above gives rise to further module
representations over $R$, referred to as the \itemph{Knuth duals} of $\theta$ in
\cite[\S 4]{ask2}.
Among these, we will need the module representation $W^* \xto{\theta^\bullet}
\Hom(V,M^*)$ given by
\begin{align*}
  x (\psi\, \theta^\bullet) & = \Bigl(M\to R, \,\,a\mapsto(x(a\theta))\psi\Bigr).
  & (\psi \in W^*, x\in V)
\end{align*}

Two module representations $M\xto\theta\Hom(V,W)$ and $\tilde M
\xto{\tilde\theta} \Hom(\tilde V,\tilde W)$ over $R$
with bilinear maps $*$ and $\tilde *$ are \emph{isotopic} (written
$\theta\approx \tilde\theta$) if there
exists a triple of module isomorphisms $(M\xto\nu\tilde M,V\xto\phi\tilde
V,W\xto\psi\tilde W)$ such that $(x * a)\psi = (x\phi) \,\tilde*\, (a\nu)$ for all
$a\in M$ and $x\in V$.
The \emph{alternating hull} $\Lambda(\theta)$ of a module representation
$M\xto\theta\Hom(V,W)$ over $R$ is the module representation
\[
  V\oplus M \to \Hom(V\oplus M, W), \quad
  (x,a)\mapsto \Bigl((x',a')\mapsto x'(a\theta) - x(a'\theta)\Bigr).
\]

Let $M\xto\theta\Hom(V,W)$ be a module representation over $R$.
We say that $\theta$ is \emph{finite free} if each of $M$, $V$, and $W$ is
free of finite rank.
Writing $\ell$, $d$, and $e$ for these ranks,
each choice of bases of all the modules involved gives rise to an isotopy
from $R^\ell \to \Mat_{d\times e}(R) = \Hom(R^d,R^e)$ to $\theta$.
If $\theta$ is finite free, then so is $\theta^\bullet$ and, moreover,
$\theta\approx\theta^{\bullet\bullet}$.
Finally, we say that $\theta$ is \emph{immersive} if it is injective, finite
free, and if $\Coker(\theta)$ is a free $R$-module.
(This notion did not feature in the work cited above.)
In that case, after choosing bases as above
and identifying $\Mat_{d\times e}(R) = R^{de}$,
the module representation $\theta$ gives rise to a closed immersion
$\AA_R^\ell \to \AA_R^{de}$.
If $\theta$ is immersive, then so are ${}^m\theta$ and
$\Lambda(\theta^\bullet)^\bullet$;
this is clear in the former case and follows from \cite[Remark~7.10]{ask2} and
$\theta \approx \theta^{\bullet\bullet}$ in the latter case.

If no ring is specified, then all module representations
are understood to be over $\ZZ$.
Note that a module representation over $\ZZ$ (or any PID) is immersive if and
only if it is injective, finite free, and has torsion-free cokernel.

\subsection{Average sizes of kernels}
\label{ss:ask}

Let $M\xto\theta\Hom(V,W)$ be a module representation over a ring $R$.
Suppose that $M$ and $V$ are finite as sets.
As in \cite[\S 1]{ask2},
the \emph{\underline a}verage \emph{\underline s}ize of the
\emph{\underline k}ernel of the linear maps parameterised by~$\theta$ 
is $\ask (\theta) = \frac 1 {\card M} \sum_{a\in M} \card{\Ker(a\theta)}$.
Let $\theta$ be a finite free module representation over $\ZZ$.
Define $\asksymb(\theta) =
\Bigl(\ask\Bigl(\theta^{\FF_q}\Bigr)\Bigr)_{q\in\PrimePowers} \in \cQ$.
This quantity only depends on the isotopy class of $\theta$,
and the image of $\asksymb(\theta)$ in $\cQo$ only depends on the isotopy
class of $\theta^{\QQ}$ over $\QQ$.

\begin{lemma}[{Cf.\ \cite[Eqn~(5.1)]{ask2}}]
  \label{lem:ask_scheme}
  Let $M\xto\theta\Hom(V,W)$ be a finite free module representation.
  Let~$C$ be the scheme representing the functor
  $R\leadsto \{ (x,a)\in (V\otimes R) \times (M\otimes R) : x(a\theta^R) =
  0\}$.
  Then $\asksymb(\theta) = \q^{-\rank(M)}\cdot \cntsymb C$ in $\cQ$.
\end{lemma}

When investigating $\asksymb(\theta)$ in $\cQo$, it suffices to consider
immersive module representations.

\begin{lemma}
  \label{lem:assume_immersive}
  Let $M\xto{\theta}\Hom(V,W)$ be a finite free module representation over $\ZZ$.
  Then there exists an immersive module representation $\tilde
  M\xto{\tilde\theta}\Hom(V,W)$ over $\ZZ$ such that
  $\asksymb({}^m\theta) = \asksymb({}^m\tilde\theta)$ in~$\cQo$ for all $m\ge 0$.
\end{lemma}
\begin{proof}
  Let $M/\Ker(\theta) \xto{\bar\theta} \Hom(V,W)$ be induced by $\theta$.
  As $M/\Ker(\theta)$ is torsion-free, $\bar\theta$ is finite free.
  Clearly, $\asksymb({}^m\theta) = \asksymb({}^m\bar\theta)$ (in $\cQ$) for
  all $m\ge 0$ so we may assume
  that $\theta$ is the inclusion of a submodule $M$ into $\Hom(V,W)$.
  Let $\tilde M/M$ be the torsion submodule of $\Hom(V,W)/M$.
  Let $\tilde\theta$ be the inclusion of $\tilde M$ into $\Hom(V,W)$;
  note that $\tilde\theta$ is immersive.
  Then $\tilde M/M$ has finite cardinality, $N$ say, and
  $\tilde M/M \otimes \FF_q = 0$ for $q\in \PrimePowers$ with
  $\gcd(q,N) = 1$.
  For such $q$, the natural map $M\otimes \FF_q \to \tilde M\otimes \FF_q$ is
  then onto and the claim follows.
\end{proof}

\subsection{Counting linear orbits of unipotent groups}
\label{ss:orbits_via_ask}

\paragraph{From module representations to abelian groups.}
Let $M\xto\theta\Mat_{d\times e}(\ZZ)$ be a finite free
module representation.
Define ${\mathsf M}_\theta$ to be the group scheme representing $R\leadsto
M\otimes R$ (additive group).
Let~$*$ be the bilinear map associated with $\theta$ as in \S\ref{ss:mreps}.
The action of $M$ on $\ZZ^{d+e} = \ZZ^d \oplus \ZZ^e$ given by $(x,y) a =
(x,x*a + y)$ naturally extends to an action of ${\mathsf M}_\theta$ on $\AA^{d+e}$.
The second proof of \cite[Lem.~2.1]{ask} shows that
\begin{equation}
  \cntsymb (\AA^{d+e} / \mathsf M_\theta) = \q^e \cdot \asksymb(\theta)
  \label{eq:ask_as_orbits}
\end{equation}
in $\cQ$.
(We can argue as in the proof of Lemma~\ref{lem:assume_immersive} to reduce to
the case that $\theta$ is injective as in the setting of
\cite{ask}.)
Since every finite free module representation is isotopic to one of the form
$M\to\Mat_{d\times e}(\ZZ)$, up to a factor of the form $\q^e$, the quantities
$\asksymb(\theta)$ attached to finite free module representations $\theta$
enumerate orbits associated with linear actions of group schemes.
If $\theta$ is immersive, then we may view ${\mathsf M}_\theta$ as a
closed (and commutative) subgroup scheme of~$\Uni_{d+e}$.
Using Lemma~\ref{lem:assume_immersive}, we conclude the following.

\begin{cor}
  \label{cor:ask_as_orbits}
  Let $\theta$ be a finite free module representation.
  \begin{enumerate*}
  \item \label{cor:ask_as_orbits1}
    If $\theta$ is immersive,  then there exist $m,n\ge 0$ and a closed 
    commutative subgroup scheme $G\le\Uni_n$ with $\orbsymb(\AA^n/G) = \q^m\cdot
    \asksymb(\theta)$ in $\cQ$.
  \item \label{cor:ask_as_orbits2}
    In any case, there exist $m,n\ge 0$ and a commutative algebraic subgroup
    $G_0\le\Uni_n\otimes \QQ$ with $\orbsymb(\AA^n_\QQ/G_0) = \q^m\cdot
    \asksymb(\theta)$ in $\cQo$.
    \qed
    \end{enumerate*}
\end{cor}

Hence, up to a factor of the form $\q^e$, the quantities $\ask(\theta)$
enumerate linear orbits.
In the unipotent case, we obtain the following converse.

\paragraph{From unipotent groups to module representations.}
Let $G_0 \le \Uni_n \otimes \QQ$ be a unipotent algebraic group over
$\QQ$.
Let $\fn_n(R)$ denote the Lie algebra of strictly upper triangular
$n\times n$ matrices over a ring $R$.
We may regard the Lie algebra $\fg_0$ of $G_0$ as a subalgebra of
$\fn_n(\QQ)$.
Let $\fg = \fg_0 \cap \fn_n(\ZZ)$, and
let $\iota$ denote the inclusion $\fg \incl \Mat_n(\ZZ)$.
This is an immersive module representation.

\begin{prop}[{\cite[Proof of Propn~8.13]{ask}}]
  \label{prop:orbits_via_ask}
  $\orbsymb(\AA^n_{\QQ}/G_0) = \asksymb(\iota)$ in $\cQo$.
\end{prop}

\subsection{Counting conjugacy classes of unipotent groups}
\label{ss:classes_via_ask}

\paragraph{From module  representations to groups I: Heisenberg-like groups.}
Let $\theta\colon M\to \Hom(V,W)$ be a module representation
with corresponding bilinear map $*\colon V\times M \to W$.
We may endow the set $M\oplus V \oplus W$ with a group structure by
formal matrix multiplication within the set $\left[\begin{smallmatrix} 1 & V &
    W \\ 0 & 1 & M \\ 0 & 0 & 1\end{smallmatrix}\right]$;
the required product $V\times M \to W$ is taken to be $*$.
Equivalently, this group is obtained as the semidirect product of $M$ acting on
$V\oplus W$, where $M$ acts via $(v,w)a = (v,v * a + w)$ ($a\in
M$, $v\in V$, $w\in W$) (cf.\ \S\ref{ss:orbits_via_ask}).
This construction of a group from $\theta$ (via $*$)
is classical; for context and variations, see \cite[\S 9.1]{Wil17}.

By base change, the above construction of a group from
$\theta$ gives rise to a group functor~$\sH_\theta$.
If $\theta$ is finite free, then $\sH_\theta$ is a group scheme and
$\sH_\theta\otimes \QQ$ is a unipotent algebraic group of class at most $2$.
If $*\colon \ZZ\times \ZZ \to \ZZ$ is the usual multiplication,
then we recover the Heisenberg group scheme~$\Uni_3$.

\begin{prop}
  \label{prop:k_H}
  Let $M\xto\theta \Hom(V,W)$ be a finite free module representation.
  Then $\consymb(\sH_\theta) = \q^{\rank(M) - \rank(V) + \rank(W) }
  \cdot \asksymb({}^2(\theta^\bullet))$ in~$\cQ$.
\end{prop}
\begin{proof}
  Combine \cite[Lem.\ 7.7]{ask2}, \cite[Thm\ 7.9]{ask2}, and
  \cite[Lem.~3.2]{ask2}.
\end{proof}

We note that ${}^2(\theta^\bullet)$ and $({}^2\theta)^\bullet$
are generally not isotopic; cf.\ \cite[Propn 4.14(i)]{ask2}.

\paragraph{From module  representations to groups II: Baer group schemes.}
Given a module representation $M\xto{\theta} \Hom(M,W)$
with corresponding bilinear map $*\colon M\times M \to W$,
we say that $\theta$ is \emph{alternating} if $a * a = 0$ for all $a \in M$.
The alternating hull (see \S\ref{ss:mreps}) of any module representation is
alternating.
Suppose that $\theta$ is alternating and finite free.
As explained in \cite[\S 2.4]{cico} (which draws upon \cite[\S 2.4.1]{SV14}),
we then obtain an associated \itemph{Baer group scheme} $\mathsf G_\theta$.
For each ring $R$, we may identify $\mathsf G_\theta(R) = (M\otimes R) \oplus
(W\otimes R)$ as sets.
With this identification, group commutators in $\mathsf G_\theta(R)$ are
characterised by the identity $[(a,y),(a',y')] = (0, a *_R a')$
for $a,a'\in M\otimes R$ and $y,y'\in W\otimes R$; here $*_R$ denotes the
bilinear map $(M\otimes R)\times (M\otimes R) \to W\otimes R$ induced by $*$
via base change.

\begin{prop}[{Cf.\ \cite[Prop.\ 2.6]{cico}}]
  \label{prop:alt_ask_via_k}
  Let $M\xto\theta\Hom(M,W)$ be an alternating finite free module
  representation.
  Then $\consymb(\mathsf G_\theta) = \q^{\rank(W)}\cdot \asksymb(\theta)$
  in $\cQ$.
\end{prop}

\paragraph{From unipotent groups to alternating module  representations.}
Let $G_0$ be a unipotent algebraic group over~$\QQ$.
Let $\fg$ be a $\ZZ$-form of the Lie algebra of $G_0$ such that
$\fg$ is free of finite rank as a $\ZZ$-module.
Let $\ad_{\fg}\colon \fg \to \End(\fg), a \mapsto (x\mapsto [x,a])$ be the
(right) adjoint representation of $\fg$;
note that this is an alternating finite free module representation.
\begin{prop}[{Cf.\ \cite[Propn~6.4]{ask2}}]
  \label{prop:k_via_ask_ad}
  $\consymb(G_0) = \asksymb(\ad_{\fg})$ in $\cQo$.
\end{prop}

One may view the preceding proposition as a consequence of
\cite[Thm~B]{O'BV15}; see \cite{ask2}.

\subsection{Reductions to low nilpotency class: proof of Proposition~\ref{prop:reduction}}
\label{ss:proof_reduction}

As we will now explain, when $G_0 \le \Uni_n\otimes \QQ$ is a unipotent
algebraic group, then the study of the $\cQo$-valued invariants
$\cntsymb (\AA^n_\QQ/G_0)$ and $\consymb(G_0)$ from \S\ref{s:arithmetic_1}
reduces to the case that $G_0$ has nilpotency class at most $2$.
(In the case of $\cntsymb (\AA^n_\QQ/G_0)$, we can even reduce to the
commutative case.)
While these reductions will come as no surprise to experts in the area, the
author feels that they deserve to be recorded in the following form.

\begin{lemma}[{= Proposition~\ref{prop:reduction}\ref{prop:reduction1}}]
  \label{lem:red1}
  Let $G_0 \le \Uni_n \otimes \QQ$ be a unipotent algebraic group over
  $\QQ$.
  Then there exists a commutative algebraic group $H_0 \le \Uni_{2n}\otimes
  \QQ$ of the same dimension as $G_0$ such that
  $\orbsymb(\AA^n_\QQ / G_0) = \q^{-n}\cdot \orbsymb(\AA^{2n}_\QQ/H_0)$ in $\cQo$.
\end{lemma}
\begin{proof}
  Given $G_0$, define $\iota$ as in the setting of
  Proposition~\ref{prop:orbits_via_ask}.
  Then
  $\orbsymb(\AA^n_{\QQ}/G_0) = \asksymb(\iota)$ in $\cQo$.
  Recall that $\iota$ is immersive.
  Using \S\ref{ss:orbits_via_ask} with~$\iota$ in place of $\theta$,
  construct a commutative group scheme $\mathsf M_\iota \le \Uni_{2n}$.
  By equation~\eqref{eq:ask_as_orbits}, $\orbsymb(\AA^{2n}_\QQ/\mathsf M_\iota) = \q^n
  \asksymb(\iota)$ in~$\cQo$.
  We may thus take $H_0 = \mathsf M_\iota \otimes \QQ$.
\end{proof}

\begin{lemma}[{= Proposition~\ref{prop:reduction}\ref{prop:reduction2}; cf.\ \cite[Prop.\ 2.5]{cico}}]
  \label{lem:red2}
  Let $G_0$ be a unipotent algebraic group of dimension $d$ over $\QQ$.
  Then there exists a unipotent algebraic group $H_0$ of class at most $2$ and
  dimension $3d$ such that
  $\consymb(G_0) = \q^{-d} \consymb(H_0)$ in $\cQo$.
\end{lemma}
\begin{proof}
  Given $G_0$, choose $\fg$ as in the setting of
  Proposition~\ref{prop:k_via_ask_ad}.
  Then
  $\consymb(G_0) = \asksymb(\ad_{\fg})$ in $\cQo$.
  Let $H_0 = \mathsf G_{\ad_{\fg}} \otimes \QQ$.
  By Proposition~\ref{prop:alt_ask_via_k}, $\consymb(H_0) = \q^d \asksymb(\ad_{\fg}) = \q^d
  \consymb(G_0)$ in $\cQo$.
\end{proof}

\subsection[Equations and inclusions]{Equations \& inclusions: from $\fAx$ to $\fXx$}
\label{ss:inclusions}

Recall that $\cQx$ denotes either $\cQ$ or $\cQo$.
Define the following $\ZZ[\q^{\pm 1}]$-submodules of $\cQx$:
\begin{align*}
  {}^m\fAx & :=
  \left\langle
    \asksymb({}^m\theta) : \theta\text{ is a finite free module representation
      (over $\ZZ$)}
             \right\rangle_{\ZZ[\q^{\pm 1}]},\\
  {}^m\fAx_\imm & :=
  \left\langle
    \asksymb({}^m\theta) : \theta\text{ is an immersive module representation
      (over $\ZZ$)}
             \right\rangle_{\ZZ[\q^{\pm 1}]},\\
  {}^{-}\fAx & :=
      \left\langle
               \asksymb(\theta) : \theta\text{ is an alternating
               finite free module representation}
      \right\rangle_{\ZZ[\q^{\pm 1}]}, \text{ and}\\
       \fXx & :=
         \Bigl\langle \cntsymb\! X  : X \text{ is a scheme} \Bigr\rangle_{\ZZ[\q^{\pm 1}]}.
\end{align*}

It is easy to see that these $\ZZ[\q^{\pm 1}]$-submodules of $\cQx$ are in
fact $\ZZ[\q^{\pm 1}]$-\itemph{subalgebras}.
(In the cases of ${}^m\fAx$ and ${}^m\fAx_\imm$ this follows since for finite
free module representations $\theta$ and $\tilde \theta$,
we have $\asksymb({}^m\theta)\cdot \asksymb({}^m{\tilde\theta}) =
\asksymb({}^m(\theta\oplus\tilde\theta))$;
cf.\ \cite[Lem.\ 3.1]{ask2}.)
Write $\fAx := {}^1 \fAx$ and $\fAx_\imm := {}^1 \fAx_\imm$.
Clearly, ${}^{mn}\fAx \subset {}^m\fAx$
and ${}^{mn}\fAx_\imm \subset {}^m\fAx_\imm$
for $m,n \ge 1$ and ${}^-\fAx\subset \fAx$.
By Lemma~\ref{lem:assume_immersive}, we have ${}^m\fAo = {}^m\fAo_\imm$ for
$m \ge 1$.
The following equations and inclusions between these algebras will be of
interest to us.

\begin{prop}
  \label{prop:eq_incl}
  \quad
  \begin{enumerate}[(i)]
  \item
    ${}^m\fAx_{\imm} \subset {}^m\fAx \subset \fAx \subset \fXx$ for all $m\ge 1$.
  \item
    $\displaystyle \fAx = \Bigl\langle \cntsymb(\AA^{d+e}/\mathsf
    M_{\theta}) : M\xto\theta\Mat_{d\times e}(\ZZ) \text{ is finite free}; \,
    d,e\ge 0
    \Bigr\rangle_{\ZZ[\q^{\pm 1}]}$
    and
    $\displaystyle \fAx_\imm = \Bigl\langle \cntsymb(\AA^{d+e}\!/\mathsf
    M_{\theta}) : M\xto\theta\Mat_{d\times e}(\ZZ) \text{ is immersive}; \,
    d,e\ge 0
    \Bigr\rangle_{\ZZ[\q^{\pm 1}]}$.
  \item
    $\displaystyle {}^2\fAx = \Bigl\langle \consymb(\mathsf H_\theta) :
    \theta\text{ is a finite free module representation}
    \Bigr\rangle_{\ZZ[\q^{\pm 1}]}$.
  \item
    $\displaystyle {}^-\fAx = \Bigl\langle \consymb(\mathsf G_\theta) :
    \theta\text{ is an alternating finite free module
      representation}\Bigr\rangle_{\ZZ[\q^{\pm 1}]}$.
  \item
  \label{cor:eq_incl5}
    ${}^2\fAx \subset {}^-\fAx$.
  \item
    We have (with an underline added solely for emphasis)
    \begin{align*} \fAo
      &=
        \Bigl\langle \cntsymb (\AA^n_\QQ / G_0)  : G_0\le \Uni_n\otimes
        \QQ \text{ is a unipotent algebraic group, } n\ge 0\Bigr\rangle_{\ZZ[\q^{\pm 1}]}
      \\&=
      \Bigl\langle \cntsymb (\AA^n_\QQ / G_0)  : G_0\le \Uni_n\otimes
      \QQ \text{ is a \underline{commutative} algebraic group, } n\ge 0\Bigr\rangle_{\ZZ[\q^{\pm 1}]}.
    \end{align*}
  \item
    We have (with an underline added solely for emphasis)
    \begin{align*}
      {}^-\fAo
      & =
        \Bigl\langle \consymb (G_0)  : G_0 \text{ is a unipotent
        algebraic group over } \QQ \Bigr\rangle_{\ZZ[\q^{\pm 1}]}
      \\&=
      \Bigl\langle \consymb (G_0)  : G_0 \text{ is a unipotent
      algebraic group \underline{of class $\le 2$} over } \QQ \Bigr\rangle_{\ZZ[\q^{\pm 1}]}.
    \end{align*}
  \end{enumerate}
\end{prop}
\begin{proof}
  \begin{enumerate*}
  \item
    The first two inclusions are clear and the third follows from Lemma~\ref{lem:ask_scheme}.
  \item
    This is a consequence of equation~\eqref{eq:ask_as_orbits}.
  \item
    This follows from Proposition~\ref{prop:k_H}
    and the fact that for each finite free module representation $\theta$,
    the dual $\theta^\bullet$ is also finite free and
    $\theta^{\bullet\bullet}\approx \theta$.
  \item
    This follows from Proposition~\ref{prop:alt_ask_via_k}.
  \item
    Let $\theta\colon M\to\Hom(V,W)$ be a finite free module representation.
    By \cite[Thm 7.9 and Lem.\ 3.2]{ask2},
    $\asksymb(\Lambda(\theta)) = \q^{\rank(M)-\rank(V)} \cdot
    \asksymb({}^2(\theta^\bullet))$ in $\cQ$.
    Hence,
    $\asksymb({}^2\theta) = \q^{\rank(V)-\rank(W)}\cdot
    \asksymb(\Lambda(\theta^\bullet)) \in {}^-\fAx$.
  \item
    Combine Corollary~\ref{cor:ask_as_orbits}(\ref{cor:ask_as_orbits2}),
    Proposition~\ref{prop:orbits_via_ask}, and Lemma~\ref{lem:red1}.
  \item
    Combine Proposition \ref{prop:alt_ask_via_k}, Proposition \ref{prop:k_via_ask_ad},
    and Lemma~\ref{lem:red2}.
  \end{enumerate*}
\end{proof}

In particular, we obtain the following chain of equations and inclusions:
\begin{align}
  \fX
  &\supset \phantom{{}^-}\fA  \,=\,
    \Bigl\langle \cntsymb(\AA^{d+e}/\mathsf
    M_{\theta}) : M\xto\theta\Mat_{d\times e}(\ZZ) \text{ is finite free}
    \Bigr\rangle_{\ZZ[\q^{\pm 1}]} \nonumber \\
   & \supset {}^-\fA  \,=\, \Bigl\langle \consymb (\mathsf G_\theta)  : \theta
     \text{ is alternating and finite free} \Bigr\rangle_{\ZZ[\q^{\pm 1}]}\nonumber\\
  & \supset {}^{\phantom.2}\fA \supset {}^{\phantom.2}\fA_\imm.
  \label{eq:inclusions}
\end{align}

While the author does not know if the inclusion ${}^2\fA_\imm \subset \fX$
is strict, Theorem~\ref{thm:main_density} will imply that ${}^2\fA_\imm$ and $\fX$ have
the same closure within the completion $\cQ_\q$ of $\cQ$ constructed in~\S\ref{ss:topology}.
Theorem~\ref{thm:main_congruence} then follows easily from this; see
\S\ref{ss:combining}.

\begin{rem}
  In \cite[\S 0.2]{BB03},
  the subring $\mathsf{CMot}^+$ of $\cZ$ additively generated by all
  $\cntsymb\! X$ as $X$ ranges over schemes is called the ring of
  \itemph{effective combinatorial motives}.
  As we will explain and exploit in~\S\ref{s:proof},
  our ring $\fX$ lies between $\mathsf{CMot}^+$ and the ring $\mathsf{CMot}$
  of \itemph{combinatorial motives} from~\cite{BB03}.
\end{rem}

\section[Main results]{Main results: Theorems~\ref{thm:main_congruence} and \ref{thm:main_density}}
\label{s:proof}

Recall the definitions of ${}^m\fA_\imm$ and $\fX$ from \S\ref{ss:inclusions}.
In this section, we prove the following.
\begin{thmabc}
  \label{thm:main_density}
  Let $m \ge 1$.
  Then $\overline{{}^m\fA_\imm} = \overline{\fX}$ in $\cQ_\q$.
\end{thmabc}

As we will see, Theorem~\ref{thm:main_density} is essentially a stronger
form of Theorem~\ref{thm:main_congruence}.
The role of immersive (rather than merely finite free) module representations
in our arguments is arguably subtle: they allow us to
view the $M_i$ in Theorem~\ref{thm:main_congruence} as closed subgroup schemes
of the $\Uni_{d_i}$; see Remark~\ref{rem:why_immersive}.

\subsection{Rank loci and limits of average sizes of kernels}
\label{ss:rank_loci_and_ask}

Let $M\xto\theta\Hom(V,W)$ be a finite free module representation (over $\ZZ$).
Let $M$, $V$, and $W$ have ranks $\ell$, $d$, and $e$, respectively.
By choosing a $\ZZ$-basis of $M$, we may identify $M = \ZZ^\ell$.
Let $D_i^\theta$ be the subscheme of $\AA^\ell$ representing the functor
\[
  R\leadsto \left\{a \in M\otimes R : \bigwedge\!{}^{i+1} \Bigl(a\theta^R\Bigr) = 0\right\}.
\]

\begin{rem}
  The $D_i^\theta$ are simply the preimages in $\AA^\ell$ of the
  \itemph{degeneracy loci} of the collections of linear maps
  $\AA^d\to \AA^d$ parameterised by $\theta$;
  cf.\ \cite[Ch.\ 14]{Ful98}.
  We can obtain an explicit description of $D_i^\theta$ by replacing $\theta$ by
  an isotopic copy $\ZZ^\ell \to \Mat_{d\times e}(\ZZ)$.
  In this setting, $\theta$ is equivalently described by a matrix of linear
  forms $A(X) \in \Mat_{d\times e}(\ZZ[X_1,\dotsc,X_\ell])$ characterised by
  $a\theta = A(a)$ for $a\in \ZZ^\ell$.
  The scheme $D_i^\theta$ is the subscheme of $\AA^\ell$ defined by the
  vanishing of the $(i+1)\times(i+1)$ minors of $A(X)$.
\end{rem}

Let $V_i^\theta = D_i^\theta \setminus D_{i-1}^\theta$ for $i > 0$ and
$V_0^\theta = D_0(\theta)$.
For our purposes, the most important property of the $V_i^\theta$ is that 
for each field $K$, we have
$V_i^\theta(K) = \left\{ a\in M\otimes K : \rank(a\theta^K) = i\right\}$.
This allows us to express average sizes of kernels derived from $\theta$ in
terms of the $V_i^\theta$.

\begin{lemma}[{Cf.\ \cite[\S 2.1]{ask}}]
  \label{lem:ask_via_rank_loci}
  $\displaystyle \asksymb  ({}^m \theta) = \q^{-\ell} \sum_{i=0}^d
  \cntsymb\! V_i^\theta \cdot \q^{m(d-i)}$ in $\cQ$.
\end{lemma}
\begin{proof}
  Fix $q\in\PrimePowers$.
  Write $\theta_q = \theta^{\FF_q}$ and
  note that $({}^m\theta)_q = {}^m(\theta_q)$.
  We partition $M\otimes \FF_q$ into the sets $Z_i := V_i^\theta(\FF_q)$ for
  $i=0,\dotsc,d$.
  For each $a\in Z_i$, since $\rank(a\theta_q) = i$, we obtain $\rank(a
  ({}^m\theta_q)) = mi$ and thus $\card{\Ker(a ({}^m\theta_q))} = q^{m(d-i)}$.
  We therefore find that
  $\ask ({}^m\theta_q) = q^{-\ell} \sum_{i=0}^d \card{Z_i} \cdot
  q^{m(d-i)}$.
\end{proof}

Write $V^\theta_\max := V^\theta_d$.
The following simple observation is key: while $\asksymb(\theta)$ might be too
coarse an invariant for us to recover any of the $\cntsymb\! V_i^\theta$ from it,
we \itemph{can} recover $\cntsymb\! V^\theta_\max$ from the $\q$-adic limit of the
sequence $\asksymb ({}^m\theta)$ as $m\to \infty$.

\begin{cor}
  \label{cor:limit}
  $\displaystyle \lim_{m\to \infty}\asksymb ({}^m\theta) = \q^{-\ell} \cdot
  \cntsymb\! V_\max^\theta$ in $\cQ_\q$.
\end{cor}
\begin{proof}
  For $i = 0,\dotsc, d-1$, the summand
  $\cntsymb\! V_i^\theta \cdot \q^{m(d-i)}$
  in Lemma~\ref{lem:ask_via_rank_loci} is divisible by $\q^m$ in $\cZ_\q$.
  Therefore, each such summand tends to zero in $\cZ_\q$ as $m\to \infty$.
\end{proof}

\begin{cor}
  \label{cor:V_in_A}
  Let $m\ge 1$.
  The closure of ${}^m\fA$ (resp.\ ${}^m\fA_\imm$)
  in $\cQ_\q$ contains $\cntsymb V^\theta_\max$ for each finite free
  (resp.\ immersive) module representation $\theta$.
\end{cor}
\begin{proof}
  $\asksymb({}^{mn}\theta) = \asksymb ({}^m({}^n\theta)) \in {}^m\fA$
  and
  $\lim\limits_{n\to \infty}\asksymb ({}^{mn}\theta) = \q^{-\ell} \cdot
  \cntsymb\! V_\max^\theta$ in $\cQ_\q$.
\end{proof}

\subsection{Consequences of the work of Belkale and Brosnan}
\label{ss:belkale_brosnan}

Let $\Gamma$ be a (finite simple) graph.
Let $v_1,\dotsc,v_n$ be the distinct vertices of $\Gamma$.
We use $\sim$ to signify adjacency between vertices.
Define a module of matrices
\[
  M_\Gamma = \left\{ x \in \Mat_n(\ZZ) : x = x^\top \text{ and } x_{ij} = 0 \text{
    whenever } v_i \sim v_j \right\}.
\]
Note that $M_\Gamma$ is a free $\ZZ$-module of rank $\binom{n+1}2 - m$, where
$m$ denotes the number of edges of $\Gamma$.
Let $\gamma = \gamma(\Gamma)$ denote the inclusion $M_\Gamma \incl \Mat_n(\ZZ)$.
It is easy to see that $\gamma$ is an immersive module representation.
The isotopy type of $\gamma$ does not depend on the chosen ordering of the
vertices of $\Gamma$.

The scheme $V^\gamma_\max$ (as defined in \S\ref{ss:rank_loci_and_ask})
coincides with $Z_\Gamma$ in the notation of \cite[\S\S 0.3, 3.2]{BB03}.
Let $R = \ZZ[\q^{\pm 1},\, (1 - \q^n)^{-1} \,\,\,(n\ge 1)]$
and $\tilde\cZ = \cZ \otimes_{\ZZ[\q]} R$.
Each $\q^n-1$ for $n\ge 1$ is regular in $\cZ$.
We may thus regard $\cZ$ as a subring of $\tilde\cZ$
and $\tilde\cZ$ as a subring  of $\prod_{\PrimePowers}\QQ$.
The following is one of the main results of Belkale and Brosnan~\cite{BB03}.

\begin{thm}[{Cf.\ \cite[Thm 0.5]{BB03}}]
  \label{thm:belkale_brosnan}
  Within $\tilde \cZ$, we have
  \[
    \left\langle \cntsymb X : X \text { is a scheme }
    \right\rangle_R
    =
    \left\langle \cntsymb V^\gamma_\max : \Gamma \text{ is a graph}
    \right\rangle_R.
  \]
\end{thm}

\begin{rem}
  In \cite{BB03},
  the module on the left-hand side of the preceding equation is
  denoted by $\mathsf{CMot}$ and called the ring of \itemph{combinatorial motives}.
\end{rem}

\subsection[Combining the pieces]{Combining the pieces: proofs of
  Theorem~\ref{thm:main_density} and \ref{thm:main_congruence}}
\label{ss:combining}

As $(1 - \q^n)^{-1} = \sum\limits_{k=0}^\infty \q^{kn}$ in $\cZ_\q$, we
may regard $\tilde \cZ$ as a subring of $\cQ_\q$.
Let $R_\q = \ZZ\llb \q\rrb[\q^{-1}] \subset \cQ_\q$.
Then $\ZZ[\q^{\pm 1}]\subset R\subset R_\q$
and $\overline{\ZZ[\q^{\pm 1}]} = \overline R = \overline{R_\q} = R_\q$ in $\cQ_\q$.
The following is elementary.

\begin{lemma}
  \label{lem:submodule_closure}
  Let $B$ be a topological ring and let $M$ be a topological $B$-module.
  Let $A\subset B$ be a dense subring.
  Let $U\subset M$ be an $A$-submodule and let $V$ be the $B$-submodule of $M$
  generated by $U$.
  Then $\bar U = \bar V$.
\qed
\end{lemma}

\begin{cor}
  $\bar\fX
  =
  \overline{
    \left\langle \cntsymb V^\gamma_\max : \Gamma \text{ is a graph}
    \right\rangle_{\ZZ[\q^{\pm 1}]}}$ within $\cQ_\q$.
\end{cor}
\begin{proof}
  Given a subring $A\subset R_\q$,
  define $A$-submodules 
   $U_A = \left\langle \cntsymb\! X : X \text { is a scheme } \right\rangle_A$
  and $V_A = \left\langle \cntsymb\! V^\gamma_\max : \Gamma \text{ is a graph}
  \right\rangle_A$ of $\cQ_\q$.
  By Lemma~\ref{lem:submodule_closure},
  $\overline{U_{\ZZ[\q^{\pm 1}]}} = \overline{U_R} = \overline{U_{R_\q}}$ and
  analogously for ``$V$''.
  Theorem~\ref{thm:belkale_brosnan} shows that $U_R = V_R$ in $\cQ_\q$.
\end{proof}

\begin{proof}[Proof of Theorem~\ref{thm:main_density}]
  We know from  Proposition~\ref{prop:eq_incl} that ${}^m \fA_\imm \subset
  \mathfrak X$.
  By Corollary~\ref{cor:V_in_A}, for each immersive
  module representation $\theta$, we have
  $\cntsymb\! V^\theta_\max\in\overline{{}^m\mathfrak A_\imm}$.
  Now apply the preceding corollary.
\end{proof}

\begin{proof}[Proof of Theorem~\ref{thm:main_congruence}]
  Equation~\eqref{eq:inclusions} and Theorem~\ref{thm:main_density} together imply
  that $\fA_\imm$ and ${}^-\fA$ are dense in $\overline{\fX}$ within
  $\cQ_\q$.
  Using our descriptions of $\fA_\imm$ and ${}^-\fA$ in terms of numbers of
  linear orbits and class numbers, respectively (see
  Proposition~\ref{prop:eq_incl}), Theorem~\ref{thm:main_congruence} follows
  from Lemma~\ref{lem:closure}.
\end{proof}

\begin{rem}
  \label{rem:why_immersive}
  The restriction to immersive module representations (via $\fA_\imm$) in the
  preceding proof is what allows us to view each $M_i$ in
  Theorem~\ref{thm:main_congruence} as a closed subgroup scheme
  of $\Uni_{d_i}$.
  If we were to instead consider all finite free module representations (via
  $\fA$), then each $M_i$ would still come with an associated homomorphism
  $M_i \to \Uni_{d_i}$ (hence an action on $\AA^{d_i}$), but this would not
  necessarily be a closed immersion.
  The conclusion of Theorem~\ref{thm:main_congruence} would otherwise remain
  unaffected.
\end{rem}

\subsection{Towards an explicit Theorem~\ref{thm:main_congruence} (mod \cite{BB03})}

In this final subsection, we briefly sketch a slightly more explicit way of
deriving Theorem~\ref{thm:main_congruence} from
Theorem~\ref{thm:belkale_brosnan}.
This is obtained by unravelling our arguments from above.
First, let $M\xto\theta\Hom(V,W)$ be any finite free module representation.
Write $\ell$, $d$, and $e$ for $\rank(M)$, $\rank(V)$, and $\rank(W)$,
respectively.

\begin{lemma}
  \label{lem:mth_powers}
  Let $m\ge 1$. Then:
  \begin{enumerate}[(i)]
  \item
    \label{lem:mth_powers1}
    $\asksymb({}^m\theta) = \q^{-me} \cdot \cntsymb(\AA^{m(d+e)} / \mathsf
    M_{{}^m\theta})$.
  \item
    \label{lem:mth_powers2}
    $\asksymb({}^{2m}\theta) = \q^{m(d-e)-\ell}\cdot  \consymb(\mathsf G_{\Lambda(({}^m\theta)^\bullet)})$.
  \end{enumerate}
\end{lemma}
\begin{proof}
  \begin{enumerate*}
  \item This follows from equation~\eqref{eq:ask_as_orbits}.
  \item
    We have
    $\asksymb({}^2\theta)  = \q^{d-e}\cdot \asksymb(\Lambda(\theta^\bullet))$
    (see the proof of Proposition~\ref{prop:eq_incl}\ref{cor:eq_incl5}).
    Hence, using Proposition~\ref{prop:alt_ask_via_k}, we obtain
  \end{enumerate*}
  $\asksymb({}^{2m}\theta) = \asksymb({}^2({}^m\theta)) = \q^{m(d-e)}\cdot \asksymb(\Lambda(({}^m\theta)^\bullet))
  = \q^{m(d-e)-\ell}\cdot \consymb(\mathsf G_{\Lambda(({}^m\theta)^\bullet)})$.
\end{proof}

As abstract group schemes, $\mathsf M_\theta \approx \mathsf
M_{{}^m\theta}$ for all $m\ge 1$.
Suppose that $\theta$ is immersive so that we may view
$\mathsf M_{{}^m\theta}$ as a closed subgroup scheme of $\GL_{m(d+e)}$.
Within $\GL_{m(d+e)}$, the group scheme $\mathsf M_{{}^m\theta}$ is then conjugate to
\[
  R\leadsto
  \left\{
    \diag\Bigl(
    \Bigl[\begin{smallmatrix} 1_d & a\theta^R \\ 0_{e\times d} & 1_e\end{smallmatrix}\Bigr],
    \dotsc,
    \Bigl[\begin{smallmatrix} 1_d & a\theta^R \\ 0_{e\times d} & 1_e\end{smallmatrix}\Bigr]\Bigr)
    : a \in M\otimes R
  \right\}.
\]

For each ring $R$, the underlying set of the Baer group scheme
$\mathsf G_{\Lambda(\theta^\bullet)}(R)$ is
$(W^*\otimes R)\times (V\otimes R)\times (M^*\otimes R)$.
We may naturally regard the right factor $M^*\otimes R$ as a central subgroup
of $\mathsf G_{\Lambda(\theta^\bullet)}(R)$.
We leave it to the reader to verify that for each $m\ge 1$,
the group $\mathsf G_{\Lambda(({}^m\theta)^\bullet)})(R)$ is the $m$-fold central
power of $\mathsf G_{\Lambda(\theta^\bullet)}(R)$ in which the copies of $M^*\otimes
R$ are identified.

Now, let the setting be as in Theorem~\ref{thm:main_congruence}.
By Theorem~\ref{thm:belkale_brosnan}, there
are graphs $\Gamma_1,\dotsc,\Gamma_r$ and rational functions
$h_1(X),\dotsc,h_r(X)\in S := \ZZ[X^{\pm 1}; (1-X^n)^{-1} \,\,(n\ge 1)]$ such that
$\cntsymb\! Y = \sum_{i=1}^r h_i(\q) \cdot \cntsymb\! V^{\gamma_i}_{\max}$;
here, $\gamma_i\colon M_{\Gamma_i}\incl \Mat_{n_i}(\ZZ)$ denotes the (immersive)
module representation attached to $\Gamma_i$ as in \S\ref{ss:belkale_brosnan}.
Let $\Gamma_i$ have $m_i$ edges so that $M_{\Gamma_i}$ has
rank $\ell_i := \binom{n_i+1} 2 - m_i$ as a $\ZZ$-module.
For $m\ge 1$, define
\[
  H_m := \sum_{i=1}^r h_i(\q)\q^{\ell_i} \cdot \ask({}^m\gamma_i).
\]
By Corollary~\ref{cor:limit}, $\lim\limits_{m\to \infty}H_m = \cntsymb\! Y$ in $\cQ_\q$.
For each $m\ge 1$, using Lemma~\ref{lem:mth_powers}, we obtain (explicit)
$f_i^m(X),g_i^m(X)\in S$ such that
$H_m = \sum_{i=1}^r f_i^m(\q) \cdot \cntsymb(\AA^{2mn_i} / \mathsf M_{{}^m\gamma_i})$
and
$H_{2m} = \sum_{i=1}^r g_i^m(\q) \cdot \consymb(\mathsf G_{\Lambda(({}^m\gamma_i)^\bullet)})$.

Fix $n\ge 1$.
There exists $m\ge 1$ such that $H_m,H_{2m}\in \cZ_\q$ and $\cntsymb \! Y
\equiv H_m \equiv H_{2m} \pmod {\q^n}$.
Given $m$, the rational functions $f_i^m(X)$ (resp.\ $g_i^{m}(X)$)
can be replaced by Laurent polynomials $f_i(X)$ (resp.\ $g_i(X)$)
such that $f_i(\q)$ (resp.\ $g_i(\q)$) and $f_i^m(\q)$ (resp.\ $g_i^m(\q)$)
are sufficiently close within $\cQ_\q$.
This yields Theorem~\ref{thm:main_congruence} and in fact, it provides the
following slight improvement.

\begin{cor}
  \label{cor:r}
  The number $r$ in Theorem~\ref{thm:main_congruence} can be
  chosen to be independent of~$n$. \qed
\end{cor}

\subsection*{Acknowledgements}

I am grateful to Christopher Voll and James B.\ Wilson for discussions on the
work described in this paper,
and to the anonymous referee for helpful comments and suggestions.

{
  \def\emph{\itemph}
  \bibliographystyle{abbrv}
  \bibliography{density}
}

\vspace*{2em}

{\small
  \noindent
  School of Mathematical and Statistical Sciences\\
  University of Galway\\Ireland\\
  E-mail: \href{mailto:tobias.rossmann@universityofgalway.ie}{tobias.rossmann@universityofgalway.ie}
}

\end{document}